\documentclass[12pt,reqno]{amsart}

\usepackage[left=3.5cm,right=3.5cm]{geometry}
\usepackage{amssymb}
\usepackage{graphicx}
\usepackage{amscd}
\usepackage[pagebackref]{hyperref}
\usepackage{color}
\usepackage{tabularx}
\usepackage[table]{xcolor}
\usepackage{float}
\usepackage{graphics,amsmath,amssymb}
\usepackage{amsthm}
\usepackage{amsfonts}
\usepackage{latexsym}
\usepackage{epsf}
\usepackage{xifthen}
\usepackage{mathrsfs}
\usepackage{dsfont}
\usepackage{makecell}
\usepackage{subfig}
\usepackage{amsmath}
\allowdisplaybreaks[4]
\usepackage{listings}
\usepackage{etoolbox}
\usepackage{fancyhdr}
\usepackage{pdflscape}
\usepackage[title,toc,titletoc]{appendix}
\usepackage{enumitem}
\usepackage[noadjust]{cite}
\usepackage{tikz}
\usetikzlibrary{automata,positioning,arrows}
\usepackage{young}
\usepackage[object=vectorian]{pgfornament} 
\usepackage{lipsum,tikz}
\usepackage{multirow}
\usepackage[OT2,T1]{fontenc}
\usepackage{mathtools}

\hypersetup{
	colorlinks=true, 
	linktoc=all, 
	linkcolor=blue} 

\numberwithin{equation}{section}

\theoremstyle{plain}
\newtheorem{theorem}{Theorem}[section]
\newtheorem*{theorem*}{Theorem}

\newtheorem{lemma}[theorem]{Lemma}

\newtheorem{innercustomgeneric}{\customgenericname}
\providecommand{\customgenericname}{}
\newcommand{\newcustomtheorem}[2]{%
	\newenvironment{#1}[1]
	{%
		\renewcommand\customgenericname{#2}%
		\renewcommand\theinnercustomgeneric{##1}%
		\innercustomgeneric
	}
	{\endinnercustomgeneric}
}
\newcustomtheorem{ctheorem}{Theorem}
\newcustomtheorem{clemma}{Lemma}

\theoremstyle{definition}
\newtheorem{definition}[theorem]{Definition}

\newtheorem*{example*}{Example}
\newtheorem*{examples*}{Examples}
\newtheorem{remark}[theorem]{Remark}
\newtheorem*{remark*}{Remark}
\newtheorem*{remarks*}{Remarks}
\newtheorem*{note*}{Note}

\newtheoremstyle{named}{}{}{\itshape}{}{\bfseries}{.}{.5em}{#1\thmnote{ #3}}
\theoremstyle{named}

\makeatletter
\patchcmd{\subsection}{\bfseries}{\bfseries\boldmath}{}{}
\makeatother

\DeclareMathAlphabet{\mydutchcal}{U}{dutchcal}{m}{n}

\newcommand{\cA}{\mydutchcal{A}}
\newcommand{\cP}{\mydutchcal{P}}
\newcommand{\cD}{\mydutchcal{D}}
\newcommand{\cS}{\mydutchcal{S}}
\newcommand{\cU}{\mydutchcal{U}}

\newcommand{\even}{\mathsf{even}}

\newcommand{\UU}{\mathbf{U}}

\newcommand{\PDO}{\operatorname{PDO}}

\title[Convolutive sequences,~I]{Convolutive sequences,~I:~Through the lens of integer partition functions}

\author[S. Chern]{Shane Chern}
\address[S. Chern]{Fakult\"at f\"ur Mathematik, Universit\"at Wien, Oskar-Morgenstern-Platz 1, Wien 1090, Austria}
\email{chenxiaohang92@gmail.com, xiaohangc92@univie.ac.at}

\author[D. Eichhorn]{Dennis Eichhorn}
\address[D. Eichhorn]{Department of Mathematics, University of California Irvine, Irvine, CA 92697, USA}
\email{deichhor@math.uci.edu}

\author[S. Fu]{Shishuo Fu$^{\ast}$}
\address[S. Fu]{College of Mathematics and Statistics, Chongqing University \& Key Laboratory of Nonlinear Analysis and its Applications (Chongqing University), Ministry of Education, Chongqing 401331, P.R. China}
\email{fsshuo@cqu.edu.cn}
\thanks{$^{\ast}$Corresponding author.}

\author[J. A. Sellers]{James A. Sellers}
\address[J. A. Sellers]{Department of Mathematics, University of Minnesota Duluth, Duluth, MN 55811, USA}
\email{jsellers@d.umn.edu}

\date{}



\begin{document}
	
\sloppy

\begin{abstract}
	Motivated by the convolutive behavior of the counting function for partitions with designated summands in which all parts are odd, we consider coefficient sequences $(a_n)_{n\ge 0}$ of primitive eta-products that satisfy the generic convolutive property
	\begin{align*}
		\sum_{n\ge 0} a_{mn} q^n = \left(\sum_{n\ge 0} a_n q^n\right)^m
	\end{align*}
	for a specific positive integer $m$. Given the results of an exhaustive search of the Online Encyclopedia of Integer Sequences for such sequences for $m$ up to $6$, we first focus on the case where $m=2$ with our attention mainly paid to the combinatorics of two $2$-convolutive sequences, featuring bijective proofs for both. For other $2$-convolutive sequences discovered in the OEIS, we apply generating function manipulations to show their convolutivity. We also give two examples of $3$-convolutive sequences. Finally, we discuss other convolutive series that are not eta-products.
\end{abstract}

\maketitle

\section{Introduction}

A \emph{partition} \cite{And1998} of a positive integer $n$ is a non-increasing sequence $\lambda$ of positive integers $\lambda_1 \geq \lambda_2 \geq \dots \geq \lambda_\ell$ such that $\lambda_1 + \lambda_2 + \dots + \lambda_\ell = n$.  We say that $n$ is the \emph{weight} (or \emph{size}) of the partition $\lambda$, which we will denote by $| \lambda |$, and these positive integers $\lambda_i$ are called \emph{parts} in this partition.  As an example, $5+2+2+2+1+1$ is a partition of weight $13$ with six parts.  At times, we will utilize an alternative shorthand notation for a partition, sometimes referred to as \emph{multiplicity notation}; for example, the partition above can be written as $(1^2, 2^3, 5^1)$ where the superscripts represent the corresponding multiplicities of each of the parts in the partition.  Note that, when we use this multiplicity notation, the part sizes will be written in the smallest-to-largest order as we read from left to right, and while they are often omitted, we occasionally use the superscript $0$ for part sizes not appearing in the partition.

In 2002, Andrews, Lewis, and Lovejoy \cite{ALL2002} introduced the combinatorial objects named \emph{partitions with designated summands}, which are partitions where exactly one part of each size in the partition is marked. For example, $5'+2+2'+2+1'+1$ is a partition of $13$ with designated summands where we have marked the only part of size five, the second part of size two, and the first part of size one.  Various aspects of these partitions with designated summands have been studied in the past; see, for example, \cite{ALL2002, Chen13, Shen, Xia}.  

In \cite{ALL2002}, the authors also considered partitions with designated summands in which all parts are odd, and they denoted the number of such objects of weight $n$ by the function $\PDO(n)$.  Since then, several others have focused attention on this counting function; see \cite{BO15, ChernSellers, FuSel, Sel2024_INTEGERS}. As has been recognized by several authors over the last two decades, the $\PDO$ function satisfies a very curious \emph{convolutive} property:
\begin{equation} \label{PDO2}
	\sum_{n\ge 0} \PDO(2n)q^n = \left ( \sum_{n\ge 0} \PDO(n)q^n \right )^2.
\end{equation}
In other words, the subsequence $(\PDO(2n))_{n\ge 0}$ is the convolution of the sequence $(\PDO(n))_{n\ge 0}$ with itself.  It is notable that \eqref{PDO2} is easily proven using the well-known generating function for $\PDO(n)$, as indicated by \cite[p.~52, eq.~(1.6) and p.~63, Theorem~21]{ALL2002}. Moreover, in a recent paper, Fu and Sellers~\cite[p.~7, Theorem~3.4]{FuSel} provided a refinement of this convolution by incorporating the contribution of two extra partition statistics.

To place relations of the same nature as \eqref{PDO2} into a broader setting, in the present work we focus on generically convolutive sequences.

\begin{definition}\label{def:mconv}
	Let $m\ge 2$ be an integer. A formal power series $\sum_{n\ge 0} a_n q^n$ is said to be \emph{$m$-convolutive} if
	\begin{align}\label{eq:mconv}
		\sum_{n\ge 0} a_{mn} q^n = \left(\sum_{n\ge 0} a_n q^n\right)^m.
	\end{align}
	In the meantime, we call the sequence $(a_n)_{n\ge 0}$ an \emph{$m$-convolutive sequence}.
\end{definition}

This paper evolved out of our attempts to \emph{combinatorially} understand the $2$-convolutivity of $\PDO(n)$. That is, our initial goal was to find a combinatorial proof of \eqref{PDO2}. Thus far, such a combinatorial proof remains elusive. However, our endeavors led us to consider the more general concept of convolutivity given in Definition~\ref{def:mconv} above, with the hope that we might identify other restricted integer partition functions that are themselves convolutive.

Indeed, an exhaustive search of the Online Encyclopedia of Integer Sequences~\cite{OEIS} identified a handful of convolutive sequences. Many of these convolutive sequences can be expressed as a product of Dedekind eta functions, which often appear as generating functions in the context of \emph{integer partitions}, thereby bearing  \emph{combinatorial} and \emph{arithmetic} meanings.

Let us define
\begin{align}\label{eq:eta-product}
	\sum_{n\ge 0} P_{m_1^{\delta_1}\cdots m_{J}^{\delta_J}}(n) q^n := \prod_{j=1}^J (q^{m_j};q^{m_j})_{\infty}^{\delta_j},
\end{align}
where $m_j$ are distinct positive integers sorted in ascending order, $\delta_j$ are integer exponents, and
\begin{align*}
	(a;q)_\infty := \prod_{k\ge 0} (1-aq^k)
\end{align*}
is the standard \emph{$q$-Pochhammer symbol}. Note that the \emph{Dedekind eta function} is defined as $\eta(\tau) := q^{\frac{1}{24}}(q;q)_\infty$, where $q:=e^{2\pi i\tau}$ with $\tau\in \mathbb{H}$, the upper half complex plane.

In addition, an eta-product taking the form of \eqref{eq:eta-product} is said to be \emph{primitive} if $\gcd(m_1, \ldots, m_J) = 1$. We remark that all eta-products involved in our OEIS search are assumed to be primitive.

In what follows, we write
\begin{align*}
	f_i := (q^i;q^i)_\infty
\end{align*}  
for notational convenience.  

Using the above notation, the well-known \emph{partition function} $p(n)$, which counts the number of partitions of $n$, can be rephrased as $P_{1^{-1}}(n)$ since its generating function is
\begin{align*}
	\sum_{n\ge 0} p(n) q^n = \frac{1}{(q;q)_\infty}.
\end{align*}
Also, according to \cite[p.~52, eq.~(1.6)]{ALL2002}, the $\PDO$ function can be alternatively written as $P_{1^{-1}3^{-1}4^{1}6^{2}12^{-1}}(n)$.

In Table~\ref{tab:list}, we list the $2$- and $3$-convolutive sequences in the OEIS~\cite{OEIS}, the last of which was submitted to the OEIS as part of this project by the present authors. These sequences form the basis of the present work.

\begin{table}[ht]
	\renewcommand{\arraystretch}{2.4} 
	\caption{A collection of $2$- or $3$-convolutive sequences in the OEIS}\label{tab:list}
	\begin{center}
		\begin{tabularx}{\textwidth}{@{}>{\centering\arraybackslash}X|c|c|>{\centering\arraybackslash}X@{}} 
			\hline
			OEIS Entry & Generating function & Reference & Convolutivity\\
			\hline
			\href{https://oeis.org/A007096}{A007096} & $\dfrac{f_2^6}{f_1^4f_4^2}$ & Thms.~\ref{thm:A007096} and \ref{th:A007096} & $2$-conv.\\
			\href{https://oeis.org/A103258}{A103258} & $\dfrac{f_2f_4^2}{f_1^2f_8}$ & Thms.~\ref{thm:A103258} and \ref{th:A103258} & $2$-conv.\\
			\href{https://oeis.org/A102186}{A102186} & $\dfrac{f_4f_6^2}{f_1f_3f_{12}}$ & Thm.~\ref{th:A102186}; cf.~Ref.~\cite{ALL2002} & $2$-conv.\\
			\href{https://oeis.org/A094023}{A094023} & $\dfrac{f_6f_{10}}{f_1f_{15}}$ & Thm.~\ref{th:A094023} & $2$-conv.\\
			\href{https://oeis.org/A128128}{A128128} & $\dfrac{f_2^3f_{3}}{f_1^3f_{6}}$ & Thm.~\ref{th:A128128} & $2$-conv.\\
			\href{https://oeis.org/A098151}{A098151} & $\dfrac{f_2f_3^2}{f_1^2f_6}$ & Thm.~\ref{th:A098151} & $3$-conv.\\
			\href{https://oeis.org/A385520}{A385520} & $\dfrac{f_2f_6^3}{f_1f_3f_4f_{12}}$ & Thm.~\ref{th:A293306} & $3$-conv.\\[10pt]
			\hline
		\end{tabularx}
	\end{center}
\end{table}

The remainder of this paper is organized as follows. In Section~\ref{sec:2cons}, we address the two ``naturally-occurring'' $2$-convolutive sequences that we understand combinatorially, and we give a bijective proof of the $2$-convolutivity of each. In Section~\ref{sec:more}, we prove the $2$-convolutivity of all sequences, including the two in Section~\ref{sec:2cons}, that we have discovered in the OEIS using generating function manipulations. Next, in Section~\ref{sec:3cons}, we give two examples of $3$-convolutive sequences. After presenting these convolutive sequences related to eta-products, we pause in Section~\ref{sec:artificial} to share some side comments about how ``artificial'' a generic convolutive sequence could be. In the meantime, we still succeed in interpreting one class of artificial convolutive sequences in a bijective manner. Finally, in Section~\ref{sec:conclusion}, we discuss our process for searching for and discovering convolutive sequences in the OEIS, and then close with questions for future consideration.

\section{Combinatorics of two 2-convolutive sequences}\label{sec:2cons}

In this section, we give \emph{bijective} proofs of the $2$-convolutivity of two sequences in the OEIS~\cite{OEIS}.
\begin{description}[itemsep=2pt]
	\item[A007096] This sequence enumerates pairs of \emph{overpartitions} (i.e., partitions where the first occurrence of each different part may be overlined; see \cite{CL2004}) into odd parts. Its generating function is
	\begin{align*}
		\sum_{n\ge 0} P_{1^{-4} 2^6 4^{-2}}(n) q^n = \frac{f_2^6}{f_1^4 f_4^2}.
	\end{align*}
	
	\item[A103258] This sequence enumerates pairs of partitions where the first partition contains only odd parts with each part size occurring at most three times, while the second partition may contain unrestricted parts with each part size also occurring at most three times. Its generating function is
	\begin{align*}
		\sum_{n\ge 0} P_{1^{-2} 2^1 4^{2} 8^{-1}}(n) q^n = \frac{f_2 f_4^2}{f_1^2 f_8}.
	\end{align*}
    We remark that the above combinatorial description is not given in the corresponding OEIS entry. However, it can be seen by rewriting the generating series as
    \begin{align*}
        \frac{f_2 f_4^2}{f_1^2 f_8} = \frac{f_2f_4}{f_1f_8}\cdot \frac{f_4}{f_1},
    \end{align*}
    while
    \begin{align*}
		\prod_{k\ge 1} (1+q^{2k-1}+q^{2(2k-1)}+q^{3(2k-1)}) = \frac{f_2f_4}{f_1f_8},
	\end{align*}
	and
	\begin{align*}
		\prod_{k\ge 1} (1+q^{k}+q^{2k}+q^{3k}) = \frac{f_4}{f_1}.
	\end{align*}
\end{description}

\subsection{Notation and lemmas}

To combinatorially prove the $2$-convolutivity of these two sequences, we set forth some definitions and notation.

\begin{definition}
	We introduce the following sets of partitions:
	\begin{itemize}[itemindent=*, leftmargin=*,itemsep=3pt]
		\item $\cP$: The set of \emph{unrestricted} partitions, i.e., partitions with no restrictions on their parts;
		
		\item $\cD$: The set of \emph{strict} partitions, i.e., partitions whose parts are distinct;
		
		\item $\cP3$: The set of partitions with each different part size occurring with multiplicity at most $3$;
		
		\item $\cP_{M,a}$: The set of partitions into parts congruent to $a$ modulo $M$;
		
		\item $\cD_{M,a}$: The set of partitions into distinct parts congruent to $a$ modulo $M$;
		
		\item $\cP3_{M,a}$: The set of partitions into parts congruent to $a$ modulo $M$ with multiplicity at most $3$;
		
		\item $\cP_{o} := \cP_{2,1}$; $\cP_{e} := \cP_{2,0}$;
		
		\item $\cD_{o} := \cD_{2,1}$; $\cD_{e} := \cD_{2,0}$;
		
		\item $\cP3_{o} := \cP3_{2,1}$; $\cP3_{e} := \cP3_{2,0}$;
		
		\item $\cS$: The ``set'' of \emph{squares} $\{s^2:s\in \mathbb{Z}\}$ where we assume that $s^2$ and $(-s)^2$ are different when $s\ne 0$.
	\end{itemize}
\end{definition}

In addition, let $\cU$ be a certain set of partitions or partition tuples. We define two operations on $\cU$.

\begin{definition}
	Let $c$ be a positive integer. We denote by $c\cU$ the set of partitions or partition tuples where each part in $\cU$ is multiplied by $c$.
\end{definition}

\begin{definition}
	We denote by $[\cU]_\even$ the subset of $\cU$ consisting of all members of an even weight.
\end{definition}

\begin{remark}\label{rmk:weight-halving}
	The latter operation in the above is especially useful in our context. Let $u(n)$ count the number of objects in $\cU$ with weight $n$. To show the sequence $(u(n))_{n\ge 0}$ is $2$-convolutive, it is sufficient to construct a weight-halving bijection from $[\cU]_\even$ to $\cU\times \cU$.
\end{remark}

For classical bijections in the realm of partitions, Pak has compiled an exhaustive survey \cite{Pak2006} which we use to our advantage.

Let us begin with a standard result of Glaisher~\cite[p.~163]{Gla1883}.

\begin{lemma}[Glaisher's bijection]\label{le:Glaisher-bij}
	Let $d\ge 2$ be an integer. There is a weight-preserving bijection between the set of partitions with no part divisible by $d$ and the set of partitions with parts occurring at most $d-1$ times.
\end{lemma}

\begin{proof}
	Pak only recorded the bijection for the case where $d=2$ in \cite[p.~23, Sect.~3.2.1]{Pak2006}, while Glaisher's result was left as an exercise in \cite[p.~23, Sect.~3.2.3]{Pak2006}. However, this bijection was exhibited well in Glaisher's paper \cite[p.~162]{Gla1883}. To benefit the reader, we briefly reproduce the map from the former partition set to the latter without going into details of the proof. Let $\lambda = (1^{m_1},2^{m_2},\ldots)$ be a partition with no part divisible by $d$, written in multiplicity notation. Notice that parts that are multiples of $d$ will not appear by definition. For each $k$ not divisible by $d$, we represent $m_k$ in base $d$ as $m_k = r_hd^h + r_{h-1}d^{h-1} + \cdots + r_0d^0$, with each $0\le r_i< d$. We then map the $m_k$ parts of size $k$ in $\lambda$ to $r_i$ copies of parts of size $k d^i$ for each $i$. 
    The resulting partition thus has parts occurring at most $d-1$ times.
\end{proof}

Next, we recall \emph{Jacobi's triple product identity} \cite[p.~21, eq.~(2.2.10)]{And1998}:  
\begin{align*}
	\sum_{n=-\infty}^\infty z^n q^{n^2} = (-zq;q^2)_\infty(-z^{-1}q;q^2)_\infty(q^2;q^2)_\infty.
\end{align*}
There are many known direct bijective proofs, for example, due to Lewis \cite{Lew1984} and Vershik \cite{Ver1988}; see \cite[Sect.~6.2]{Pak2006} for Pak's survey. In particular, 
in \cite[Sect.~6.2.2]{Pak2006}, the following result is stated.

\begin{lemma}\label{lemma:theta}
	There is a bijection
	\begin{align*}
		\begin{array}{ccc}
			\cD_{o}\times \cD_{o} & \to & \cS\times \cP_{e}\\
			(\mu,\nu) & \mapsto & (s^2,\lambda)
		\end{array}
	\end{align*}
	such that
	\begin{align*}
		|\mu|+|\nu|=s^2+|\lambda|
	\end{align*}
	and that
	\begin{align*}
		s=\ell(\mu)-\ell(\nu),
	\end{align*}
	where $\ell(\pi)$ counts the number of parts in a partition $\pi$.
\end{lemma}

This bijection essentially shows Jacobi's triple product identity because the set on the left has the generating function
\begin{align*}
	\sum_{(\mu,\nu)\in \cD_{o}\times \cD_{o}}z^{\ell(\mu)-\ell(\nu)}q^{|\mu|+|\nu|} = (-zq;q^2)_\infty(-z^{-1}q;q^2)_\infty,
\end{align*}
while the set on the right has the generating function
\begin{align*}
	\sum_{(s^2,\lambda)\in \cS\times \cP_{e}}z^{s}q^{s^2+|\lambda|} = \frac{1}{(q^2;q^2)_\infty}\sum_{n=-\infty}^\infty z^n q^{n^2}.
\end{align*}

For the moment, we construct several more bijections for later use.

\begin{lemma}\label{lemma:splitP3o}
	There is a weight-preserving bijection from $\cP3_{o}$ to $\cD_{o} \times \cD_{4,2}$.
\end{lemma}

\begin{proof}
	Suppose 
	$(1^{m_1},3^{m_3}, \dots, K^{m_K}) \in \cP3_o$, where each $0 \leq m_i \leq 3$ and $K$ is odd. We map the partition $(1^{m_1},3^{m_3}, \dots, K^{m_K})$ bijectively to the pair
	\begin{align*}
		\big((1^{m_1 - 2\lfloor m_1/2 \rfloor},3^{m_3 - 2\lfloor m_3/2 \rfloor}, \dots, K^{m_K-2\lfloor m_K/2 \rfloor}), \ (2^{\lfloor m_1/2 \rfloor},6^{\lfloor m_3/2 \rfloor},\dots, 2K^{\lfloor m_K/2 \rfloor})\big),
	\end{align*}
	which is clearly in $\cD_{o} \times \cD_{4,2}$.
\end{proof}



\begin{lemma}\label{lemma:splitP}
	There is a weight-preserving bijection from $\cP$ to $\cP3 \times \cP_{4,0}$.
\end{lemma}

\begin{proof}
	Separating the parts from partitions in $\cP$ according to their residue modulo $4$, we have that $\cP = \cP_{4,1} \times \cP_{4,2} \times\cP_{4,3} \times\cP_{4,0}$. By Glaisher's bijection in Lemma~\ref{le:Glaisher-bij} with $d=4$, we have that $\cP_{4,1} \times \cP_{4,2} \times\cP_{4,3}$ is bijectively equivalent to $\cP3$, which completes the proof.
\end{proof}

\begin{lemma}\label{lemma:sstt}
	There is a weight-halving bijection from $[\cS\times \cS]_\even$ to $\cS\times \cS$.
\end{lemma}

\begin{proof}
	The following map
	\begin{align*}
		\begin{array}{ccc}
			[\cS\times \cS]_\even & \to & \cS\times \cS\\[6pt]
			(s_1^2,s_2^2) & \mapsto & (t_1^2,t_2^2)
		\end{array}
	\end{align*}
	given by
	\begin{align*}
		\begin{cases}
			t_1=\frac{1}{2}(s_1+s_2)\\
			t_2=\frac{1}{2}(s_1-s_2)
		\end{cases}
	\end{align*}
	is as claimed. Here we need to notice that $s_1$ and $s_2$ should have the same parity so as to ensure that $s_1^2+s_2^2$ is even and hence the pair $(s_1^2,s_2^2)$ belongs to $[\cS\times \cS]_\even$.
\end{proof}

\subsection{Sequence A007096}

We now give the desired bijective proof of the $2$-convolutivity of the OEIS sequence A007096, which is given by
\begin{align*}
	\sum_{n\ge 0} P_{1^{-4} 2^6 4^{-2}}(n) q^n = \frac{f_2^6}{f_1^4 f_4^2}.
\end{align*}
Note that the eta-product here can be rewritten as
\begin{align*}
	\frac{f_2^6}{f_1^4 f_4^2} = \frac{f_2^2}{f_1f_4} \cdot \frac{f_2^2}{f_1f_4} \cdot \frac{f_2}{f_1}\cdot \frac{f_2}{f_1} = (-q;q^2)_\infty\cdot  (-q;q^2)_\infty\cdot  (-q;q)_\infty\cdot  (-q;q)_\infty,
\end{align*}
which is clearly the generating function for $\cD_{o} \times \cD_{o} \times \cD \times \cD$.

\begin{theorem}\label{thm:A007096}
	There is a weight-halving bijection from
	\begin{align*}
		[\cD_{o} \times \cD_{o} \times \cD \times \cD]_\even
	\end{align*}
	to
	\begin{align*}
		\cD_{o} \times \cD_{o} \times \cD_{o} \times \cD_{o} \times \cD \times \cD \times \cD \times \cD.
	\end{align*}
	As a consequence, the numbers $P_{1^{-4} 2^6 4^{-2}}(n)$ form a $2$-convolutive sequence.
\end{theorem}

\begin{proof}
	Consider 
	\begin{align*}
		[\cD_{o} \times \cD_{o} \times \cD \times \cD]_\even.
	\end{align*}
	By separating the parts from partitions in $\cD$ according to their parity, we have that $\cD = \cD_o \times \cD_e$. Thus $[\cD_{o} \times \cD_{o} \times \cD \times \cD]_\even$ is in bijection with $[\cD_{o} \times \cD_{o} \times \cD_o \times \cD_o \times \cD_e \times \cD_e]_\even$. Applying Lemma~\ref{lemma:theta} twice, we see that this is further in bijection with $[\cS \times \cS \times \cP_e \times \cP_e \times \cD_{e} \times \cD_{e}]_\even$.
	Notice that for an element of $\cS \times \cS$ to contribute to $[\cS \times \cS \times \cP_e \times \cP_e \times \cD_{e} \times \cD_{e}]_\even$, it must actually be in $[\cS \times \cS]_\even$, so that in fact,
	\begin{align*}
		&[\cS \times \cS \times \cP_e \times \cP_e \times \cD_{e} \times \cD_{e}]_\even\\
		&\qquad = [[\cS \times \cS]_\even \times \cP_e \times \cP_e \times \cD_{e} \times \cD_{e}]_\even.
	\end{align*}
	
	Now, by applying Lemma~\ref{lemma:sstt} to $[\cS \times \cS]_\even$ and halving each part in $\cP_e \times \cP_e \times \cD_{e} \times \cD_{e}$, we have that $[\cD_{o} \times \cD_{o} \times \cD \times \cD]_\even$ is in (a weight-halving) bijection with $\cS \times \cS \times \cP \times \cP \times \cD \times \cD$. Separating the parts from partitions in $\cP$ according to their parity, we have that this is in bijection with $\cS \times \cS \times \cP_e \times \cP_e \times \cP_o \times \cP_o \times \cD \times \cD$. Applying Lemma~\ref{lemma:theta} twice, we see that this is in bijection with $\cD_o \times \cD_o \times \cD_o \times \cD_o \times \cP_o \times \cP_o \times \cD \times \cD$. Finally, applying Glaisher's bijection twice in Lemma~\ref{le:Glaisher-bij} with $d=2$, i.e., the bijection between $\cP_o$ and $\cD$, we see that this is in bijection with
	\begin{align*}
		\cD_o \times \cD_o \times \cD_o \times \cD_o \times \cD \times \cD \times \cD \times \cD,
	\end{align*}
	as desired.
	
	Finally, we conclude the $2$-convolutivity of the sequence $(P_{1^{-4} 2^6 4^{-2}}(n))_{n\ge 0}$ in light of Remark~\ref{rmk:weight-halving}.
\end{proof}

\subsection{Sequence A103258}

In this subsection, we look at the OEIS sequence A103258, which is given by
\begin{align*}
	\sum_{n\ge 0} P_{1^{-2} 2^1 4^{2} 8^{-1}}(n) q^n = \frac{f_2 f_4^2}{f_1^2 f_8}.
\end{align*}
The above eta-product can be rewritten as
\begin{align*}
	\frac{f_2 f_4^2}{f_1^2 f_8} = \frac{(q^4;q^8)_\infty}{(q;q^2)_\infty}\cdot \frac{(q^4;q^4)_\infty}{(q;q)_\infty},
\end{align*}
and as we have argued, it gives the generating function for $\cP3_{o} \times \cP3$.

\begin{theorem}\label{thm:A103258}
	There is a weight-halving bijection from
	\begin{align*}
		[\cP3_{o} \times \cP3]_\even
	\end{align*}
	to
	\begin{align*}
		\cP3_{o} \times \cP3_{o} \times \cP3 \times \cP3.
	\end{align*}
	As a consequence, the numbers $P_{1^{-2} 2^1 4^{2} 8^{-1}}(n)$ form a $2$-convolutive sequence.
\end{theorem}

\begin{proof}
	Consider
	\begin{align*}
		[\cP3_{o} \times \cP3]_\even.
	\end{align*}
	By separating the parts from partitions in $\cP3$ according to their parity, we have that $\cP3 = \cP3_o \times \cP3_e$.
	Thus $[\cP3_{o} \times \cP3]_\even$ is in bijection with $[\cP3_{o} \times \cP3_o \times \cP3_e]_\even$. Applying Lemma~\ref{lemma:splitP3o} twice, we see that this is in bijection with $[\cD_{o} \times \cD_{o} \times \cD_{4,2} \times \cD_{4,2} \times \cP3_e]_\even$. Applying Lemma~\ref{lemma:theta}, we see that this is in bijection with $[\cS \times \cP_e \times \cD_{4,2} \times \cD_{4,2} \times \cP3_e]_\even$. Now notice that tuples of even size in $\cS \times \cP_e \times \cD_{4,2} \times \cD_{4,2} \times \cP3_e$ require that our elements of $\cS$ are even, and so
	\begin{align*}
		&[\cS \times \cP_e \times \cD_{4,2} \times \cD_{4,2} \times \cP3_e]_\even\\
		&\qquad = [4\cS \times \cP_e \times \cD_{4,2} \times \cD_{4,2} \times \cP3_e]_\even.
	\end{align*}
	
	Now, by halving each part, we have that $[\cP3_{o} \times \cP3]_\even$ is in (a weight-halving) bijection with $2\cS \times \cP \times \cD_o \times \cD_o \times \cP3$. Applying Lemma~\ref{lemma:splitP}, we see that this is in bijection with $2\cS \times \cP_{4,0} \times \cD_o \times \cD_o \times \cP3 \times \cP3$. Since $\cP_{4,0} = 2\cP_e$, we may apply Lemma~\ref{lemma:theta} to see that this is in bijection with
	\begin{align*}
		&2(\cD_o \times \cD_o) \times \cD_o \times \cD_o \times \cP3 \times \cP3\\
		&\qquad = \cD_{4,2} \times \cD_{4,2} \times \cD_o \times \cD_o \times \cP3 \times \cP3.
	\end{align*}
	Applying Lemma~\ref{lemma:splitP3o} twice, we see that this is in bijection with
	\begin{align*}
		\cP3_o \times \cP3_o \times \cP3 \times \cP3,
	\end{align*}
	as desired.
\end{proof}

\section{More 2-convolutive sequences}\label{sec:more}

As mentioned earlier, we have identified additional $2$-convolutive sequences in the OEIS~\cite{OEIS} for which we have yet to obtain combinatorial proofs. We hereby discuss the $2$-convolutivity of these sequences and provide straightforward \emph{analytic} proofs. We include here the analytic proofs of the $2$-convolutivity of the two sequences from Section~\ref{sec:2cons} as well, for the sake of completeness.
\begin{description}[itemsep=2pt]
	\item[A102186] This sequence is the $\PDO$ function, which enumerates partitions with designated summands wherein all parts are odd. Its generating function is
	\begin{align*}
		\sum_{n\ge 0} P_{1^{-1}3^{-1}4^{1}6^{2}12^{-1}}(n) q^n = \frac{f_4 f_6^2}{f_1 f_3 f_{12}}.
	\end{align*}
	
	\item[A094023] This sequence enumerates triples of partitions where the first partition contains no multiples of $3$ or $5$, the second partition is into distinct multiples of $3$, and the third partition is into distinct multiples of $5$. Its generating function is
	\begin{align*}
		\sum_{n\ge 0} P_{1^{-1} 6^1 10^{1} 15^{-1}}(n) q^n = \frac{f_6 f_{10}}{f_1 f_{15}}.
	\end{align*}
    We remark that the above combinatorial description is not given in the corresponding OEIS entry. However, it can be seen by rewriting the generating series as
	\begin{align*}
		\frac{f_6 f_{10}}{f_1 f_{15}} = \frac{f_3 f_5}{f_1 f_{15}} \cdot \frac{f_6}{f_3} \cdot \frac{f_{10}}{f_5}.
	\end{align*}

    \item[A128128] This sequence enumerates triples of partitions where the first partition contains no multiples of $2$ or $3$, and the second and third partitions are into distinct parts. Its generating function is
	\begin{align*}
		\sum_{n\ge 0} P_{1^{-3} 2^3 3^{1} 6^{-1}}(n) q^n = \frac{f_2^3 f_{3}}{f_1^3 f_{6}}.
	\end{align*}
    We remark that the above combinatorial description is again not given in the corresponding OEIS entry. However, it can be seen by rewriting the generating series as
	\begin{align*}
		\frac{f_2^3 f_{3}}{f_1^3 f_{6}} = \frac{f_2 f_3}{f_1f_6}  \cdot \frac{f_2}{f_1} \cdot \frac{f_2}{f_1}.
	\end{align*}   
\end{description}


\subsection{2-Dissection identities}

In this subsection, we collect several \emph{$2$-dissection} formulas for $q$-series. 
Most of these results are recorded in Hirschhorn's monograph~\cite{Hir}.  

\begin{lemma}
	\begin{align}
		\frac{1}{f_1^2}&=\frac{f_8^5}{f_2^5f_{16}^2}+2q\frac{f_4^2f_{16}^2}{f_2^5f_8},\label{eq:f1-2}\\
		\frac{1}{f_1^4}&=\frac{f_4^{14}}{f_2^{14}f_8^4}+4q\frac{f_4^2f_8^4}{f_2^{10}}.\label{eq:f1-4}
	\end{align}
\end{lemma}

\begin{proof}
	For \eqref{eq:f1-2}, see \cite[eq.~(1.9.4)]{Hir}. For \eqref{eq:f1-4}, see \cite[eq.~(1.10.1)]{Hir}.
\end{proof}

\begin{lemma}
	\begin{align}
		\frac{f_3}{f_1^3}&=\frac{f_4^6f_6^3}{f_2^9f_{12}^2}+3q\frac{f_4^2f_6f_{12}^2}{f_2^7},\label{eq:f1f3}\\
		\frac{1}{f_1f_3}&= \frac{f_8^2f_{12}^5}{f_2^2f_4f_6^4f_{24}^2}+q\frac{f_4^5f_{24}^2}{f_2^4f_6^2f_8^2f_{12}}.\label{eq:f1f3inv}
	\end{align}
\end{lemma}

\begin{proof}
	For \eqref{eq:f1f3}, see \cite[eq.~(22.7.3) with $q \mapsto -q$]{Hir}. For \eqref{eq:f1f3inv}, see \cite[eq.~(30.12.3)]{Hir}.
\end{proof}

Let $\UU_m$ be the \emph{unitizing operator of degree $m$} defined by
\begin{align*}
	\UU_m\left(\sum_{n\ge 0} a_n q^n\right) := \sum_{n\ge 0} a_{mn} q^n.
\end{align*}

\begin{lemma}\label{le:H-f1f15}
	\begin{align}\label{eq:H-f1f15}
		\UU_2\left(\frac{1}{f_1 f_{15}}\right) = \frac{f_{6}^2f_{10}^2}{f_1^2f_3f_{5}f_{15}^2}.
	\end{align}
\end{lemma}

\begin{proof}
	We require the following fifteenth-degree modular equation \cite[p.~377, Entry~9(iv)]{Ber91}:
	\begin{align*}
		\psi(q)\psi(q^{15}) + \psi(-q)\psi(-q^{15}) = 2\frac{f_{12}^2f_{20}^2}{f_6f_{10}},
	\end{align*}
	where $\psi(q):= f_2^2/f_1$ is one of \emph{Ramanujan's classical theta functions} \cite[p.~36, Entry~22(ii)]{Ber91}. Note that
	\begin{align*}
		\UU_2\big(\psi(q)\psi(q^{15})\big) = \frac{1}{2}\big[\psi(q)\psi(q^{15}) + \psi(-q)\psi(-q^{15})\big]_{q^2\mapsto q} = \frac{f_{6}^2f_{10}^2}{f_3f_{5}}.
	\end{align*}
	Now we only need the fact
	\begin{align*}
		\UU_2\left(\frac{1}{f_1 f_{15}}\right) = \frac{\UU_2\big(\psi(q)\psi(q^{15})\big)}{f_1^2 f_{15}^2}
	\end{align*}
	to conclude the claimed result.
\end{proof}

\subsection{Analytic proofs of 2-convolutivity}

Now we are in a position to offer analytic proofs of the $2$-convolutivity for the OEIS sequences listed at the beginning of this section.

\subsubsection{Sequence A102186, the $\PDO$ function}

Our prototypical example of a $2$-convolutive sequence is  $\PDO(n)$.  A generating function proof of this fact can be found in \cite{ALL2002}, and this is also highlighted in \cite{Sel2024_INTEGERS}.  Even so, we share such a proof here for the sake of completeness.

\begin{theorem}\label{th:A102186}
	The numbers $P_{1^{-1}3^{-1}4^{1}6^{2}12^{-1}}(n)$ form a $2$-convolutive sequence.
\end{theorem}

\begin{proof}
	Recall that
	\begin{align*}
		\sum_{n\ge 0} P_{1^{-1}3^{-1}4^{1}6^{2}12^{-1}}(n) q^n = \frac{f_4 f_6^2}{f_1 f_3 f_{12}}.
	\end{align*}
	We have
	\begin{align*}
		\frac{f_4 f_6^2}{f_1 f_3 f_{12}} = \frac{f_4 f_6^2}{f_{12}}\cdot \frac{1}{f_1 f_3} \overset{\eqref{eq:f1f3inv}}{=} \frac{f_4 f_6^2}{f_{12}}\cdot \left(\frac{f_8^2f_{12}^5}{f_2^2f_4f_6^4f_{24}^2}+q\frac{f_4^5f_{24}^2}{f_2^4f_6^2f_8^2f_{12}}\right).
	\end{align*}
	Therefore,
	\begin{align*}
		\UU_2\left(\frac{f_4 f_6^2}{f_1 f_3 f_{12}}\right) = \frac{f_2 f_3^2}{f_{6}}\cdot \frac{f_4^2f_{6}^5}{f_1^2f_2f_3^4f_{12}^2} = \left(\frac{f_4 f_6^2}{f_1 f_3 f_{12}}\right)^2,
	\end{align*}
	as required.
\end{proof}

\subsubsection{Sequence A094023}

\begin{theorem}\label{th:A094023}
	The numbers $P_{1^{-1} 6^1 10^{1} 15^{-1}}(n)$ form a $2$-convolutive sequence.
\end{theorem}

\begin{proof}
	Recall that
	\begin{align*}
		\sum_{n\ge 0} P_{1^{-1} 6^1 10^{1} 15^{-1}}(n) q^n = \frac{f_6 f_{10}}{f_1 f_{15}}.
	\end{align*}
	Therefore,
	\begin{align*}
		\UU_2\left(\frac{f_6 f_{10}}{f_1 f_{15}}\right) = f_3 f_5\cdot \UU_2\left(\frac{1}{f_1 f_{15}}\right) \overset{\eqref{eq:H-f1f15}}{=} f_3 f_5\cdot \frac{f_{6}^2f_{10}^2}{f_1^2f_3f_{5}f_{15}^2} = \left(\frac{f_6 f_{10}}{f_1 f_{15}}\right)^2,
	\end{align*}
	as required.
\end{proof}

\subsubsection{Sequence A128128}

\begin{theorem}\label{th:A128128}
	The numbers $P_{1^{-3} 2^3 3^{1} 6^{-1}}(n)$ form a $2$-convolutive sequence.
\end{theorem}

\begin{proof}
	Recall that
	\begin{align*}
		\sum_{n\ge 0} P_{1^{-3} 2^3 3^{1} 6^{-1}}(n) q^n = \frac{f_2^3 f_{3}}{f_1^3 f_{6}}.
	\end{align*}
	We have
	\begin{align*}
		\frac{f_2^3 f_{3}}{f_1^3 f_{6}} = \frac{f_2^3}{f_{6}}\cdot \frac{f_{3}}{f_1^3} \overset{\eqref{eq:f1f3}}{=} \frac{f_2^3}{f_{6}}\cdot \left(\frac{f_4^6f_6^3}{f_2^9f_{12}^2}+3q\frac{f_4^2f_6f_{12}^2}{f_2^7}\right).
	\end{align*}
	Therefore,
	\begin{align*}
		\UU_2\left(\frac{f_2^3 f_{3}}{f_1^3 f_{6}}\right) = \frac{f_1^3}{f_{3}}\cdot \frac{f_2^6f_3^3}{f_1^9f_{6}^2} = \left(\frac{f_2^3 f_{3}}{f_1^3 f_{6}}\right)^2,
	\end{align*}
	as required.
\end{proof}

\subsubsection{Sequence A007096}

\begin{theorem}\label{th:A007096}
	The numbers $P_{1^{-4} 2^6 4^{-2}}(n)$ form a $2$-convolutive sequence.
\end{theorem}

\begin{proof}
	Recall that
	\begin{align*}
		\sum_{n\ge 0} P_{1^{-4} 2^6 4^{-2}}(n) q^n = \frac{f_2^6}{f_1^4 f_4^2}.
	\end{align*}
	We have
	\begin{align*}
		\frac{f_2^6}{f_1^4 f_4^2} = \frac{f_2^6}{f_4^2}\cdot \frac{1}{f_1^4} \overset{\eqref{eq:f1-4}}{=} \frac{f_2^6}{f_4^2}\cdot \left(\frac{f_4^{14}}{f_2^{14}f_8^4}+4q\frac{f_4^2f_8^4}{f_2^{10}}\right).
	\end{align*}
	Therefore,
	\begin{align*}
		\UU_2\left(\frac{f_2^6}{f_1^4 f_4^2}\right) = \frac{f_1^6}{f_2^2}\cdot \frac{f_2^{14}}{f_1^{14}f_4^4} = \left(\frac{f_2^6}{f_1^4 f_4^2}\right)^2,
	\end{align*}
	as required.
\end{proof}

\subsubsection{Sequence A103258}

\begin{theorem}\label{th:A103258}
	The numbers $P_{1^{-2} 2^1 4^{2} 8^{-1}}(n)$ form a $2$-convolutive sequence.
\end{theorem}

\begin{proof}
	Recall that
	\begin{align*}
		\sum_{n\ge 0} P_{1^{-2} 2^1 4^{2} 8^{-1}}(n) q^n = \frac{f_2 f_4^2}{f_1^2 f_8}.
	\end{align*}
	We have
	\begin{align*}
		\frac{f_2 f_4^2}{f_1^2 f_8} = \frac{f_2 f_4^2}{f_8}\cdot \frac{1}{f_1^2} \overset{\eqref{eq:f1-2}}{=} \frac{f_2 f_4^2}{f_8}\cdot \left(\frac{f_8^5}{f_2^5f_{16}^2}+2q\frac{f_4^2f_{16}^2}{f_2^5f_8}\right).
	\end{align*}
	Therefore,
	\begin{align*}
		\UU_2\left(\frac{f_2 f_4^2}{f_1^2 f_8}\right) = \frac{f_1 f_2^2}{f_4}\cdot \frac{f_4^5}{f_1^5f_{8}^2} = \left(\frac{f_2 f_4^2}{f_1^2 f_8}\right)^2,
	\end{align*}
	as required.
\end{proof}

\section{3-Convolutive sequences}\label{sec:3cons}

In the previous two sections, our attention was restricted to $2$-convolutivity. Since we have introduced the more general $m$-convolutive sequences in Definition~\ref{def:mconv}, a natural question to ask is  \textit{whether there exist any instances of higher convolutive sequences?} Fortunately, as part of our exhaustive search of the OEIS~\cite{OEIS}, we identified two $3$-convolutive sequences:
\begin{description}[itemsep=2pt]
	\item[A098151] This sequence enumerates overpartitions with no parts (either overlined or not) divisible by $3$. Its generating function is
	\begin{align*}
		\sum_{n\ge 0} P_{1^{-2}2^{1}3^{2}6^{-1}}(n) q^n = \frac{f_2 f_3^2}{f_1^2 f_6}.
	\end{align*}
    We remark that two other combinatorial descriptions of this sequence are provided in the corresponding OEIS entry, with one of them featuring Schur overpartitions~\cite[pp.~216--217, Corollary~1.4]{Lov2005}. To see how our description here works, we simply rewrite the generating series as
	\begin{align*}
		\frac{f_2 f_3^2}{f_1^2 f_6} = \frac{f_2}{f_1^2}\cdot \frac{f_3^2}{f_6}.
	\end{align*}

	\item[A385520] This sequence enumerates partitions where each even part can appear at most twice, while each odd part can appear once, thrice, or four times. Its generating function is
	\begin{align*}
		\sum_{n\ge 0} P_{1^{-1} 2^1 3^{-1} 4^{-1} 6^3 12^{-1}}(n) q^n = \frac{f_2 f_6^3}{f_1 f_3 f_4 f_{12}}.
	\end{align*}
    To see the above partition-theoretic description of this eta-product, note that 
\begin{align*}
  \frac{f_2 f_6^3}{f_1 f_3 f_4 f_{12}} &= \frac{f_2^2}{f_1f_4}\cdot\frac{f_6^2}{f_3f_{12}}\cdot\frac{f_6}{f_2} = (-q;q^2)_{\infty}(-q^3;q^6)_{\infty}\cdot \frac{(q^6;q^6)_{\infty}}{(q^2;q^2)_{\infty}} \\
  &=\prod_{i\ge 1}(1+q^{2i-1}+q^{6i-3}+q^{8i-4})\prod_{j\ge 1}(1+q^{2j}+q^{4j}).
\end{align*}
\end{description}

\begin{remark}
    Our search of the OEIS actually revealed four $3$-convolutive sequences, but they form two pairs of dual sequences in the following sense. For a given sequence $(a_n)_{n\ge 0}$, we let $a'_n:=(-1)^na_n$ and call the sequence $(a'_n)_{n\ge 0}$ the \emph{dual} sequence of $(a_n)_{n\ge 0}$. Noticing that $(-1)^{3n}=(-1)^n$ for any integer $n$, it becomes evident that
    $$a_{3n}=\sum_{0\le r,s\le n}a_ra_sa_{n-r-s}$$
    holds if and only if
    $$a'_{3n}=\sum_{0\le r,s\le n}a'_ra'_sa'_{n-r-s}$$
    holds. That is to say, a sequence is $3$-convolutive if and only if its dual sequence is $3$-convolutive. The sequences A132002 and A293306 from the OEIS are the dual sequences of A098151 and A385520, respectively, rendering both of them $3$-convolutive as well.
\end{remark}

\subsection{3-Dissection identities}

Recall that we have introduced one of Ramanujan's theta functions $\psi(q)$ for the proof of Lemma~\ref{le:H-f1f15}. Now we need another theta function $\varphi(q)$ and we note from \cite[eqs.~(1.5.8) and (1.5.9)]{Hir} that
\begin{align*}
	\varphi(-q) = \frac{f_1^2}{f_2},\qquad\qquad \psi(-q) = \frac{f_1 f_4}{f_2}.
\end{align*}
The following \emph{$3$-dissection} identities are useful in our analysis.

\begin{lemma}
	\begin{align}
		\frac{1}{\varphi(-q)}&=\frac{\varphi(-q^9)^3}{\varphi(-q^3)^4}\left(1+2qw(q^3)+4q^2 w(q^3)^2\right),\label{eq:1/-phi--3dis}\\
		\frac{1}{\psi(-q)}&=\frac{\psi(-q^9)^3}{\psi(-q^3)^4}\left(\frac{1}{w(-q^3)^2}+\frac{q}{w(-q^3)}+q^2\right)\label{eq:1/psi--3dis},
	\end{align}
	where
	\begin{equation}
		w(q)=\frac{f_1f_6^3}{f_2f_3^3}.\label{def:w(q)}
	\end{equation}
\end{lemma}

\begin{proof}
	We start with \cite[eq.~(14.3.4)]{Hir} and substitute $\omega q$ and $\omega^2 q$ for $q$ where $\omega$ is a primitive cubic root of unity. Multiplying the two results, we obtain \eqref{eq:1/-phi--3dis}. For \eqref{eq:1/psi--3dis}, we use \cite[eq.~(14.3.5)]{Hir} and conduct the same analysis.
\end{proof}

\subsection{Analytic proofs of 3-convolutivity}

Now we are ready to prove the $3$-convolutivity for the two sequences discussed in this section in view of their generating functions.

\subsubsection{Sequence A098151}

\begin{theorem}\label{th:A098151}
	The numbers $P_{1^{-2}2^{1}3^{2}6^{-1}}(n)$ form a $3$-convolutive sequence.
\end{theorem}

\begin{proof}
	Recall that
	\begin{align*}
		\sum_{n\ge 0} P_{1^{-2}2^{1}3^{2}6^{-1}}(n) q^n = \frac{f_2 f_3^2}{f_1^2 f_6}.
	\end{align*}
	We have
	\begin{align*}
		\frac{f_2 f_3^2}{f_1^2 f_6} = \frac{f_3^2}{f_6}\cdot \frac{1}{\varphi(-q)} \overset{\eqref{eq:1/-phi--3dis}}{=} \frac{f_3^2}{f_6}\cdot \left(\frac{\varphi(-q^9)^3}{\varphi(-q^3)^4}\left(1+2qw(q^3)+4q^2 w(q^3)^2\right)\right).
	\end{align*}
	Therefore,
	\begin{align*}
		\UU_3\left(\frac{f_2 f_3^2}{f_1^2 f_6}\right) = \frac{f_1^2}{f_2}\cdot \frac{\varphi(-q^3)^3}{\varphi(-q)^4} = \left(\frac{f_2 f_3^2}{f_1^2 f_6}\right)^3,
	\end{align*}
	as required.
\end{proof}

\subsubsection{Sequence A385520}

\begin{theorem}\label{th:A293306}
	The numbers $P_{1^{-1} 2^1 3^{-1} 4^{-1} 6^3 12^{-1}}(n)$ form a $3$-convolutive sequence.
\end{theorem}

\begin{proof}
	Recall that
	\begin{align*}
		\sum_{n\ge 0} P_{1^{-1} 2^1 3^{-1} 4^{-1} 6^3 12^{-1}}(n) q^n = \frac{f_2 f_6^3}{f_1 f_3 f_4 f_{12}}.
	\end{align*}
	We have
	\begin{align*}
		\frac{f_2 f_6^3}{f_1 f_3 f_4 f_{12}} = \frac{f_6^3}{f_3 f_{12}}\cdot \frac{1}{\psi(-q)} \overset{\eqref{eq:1/psi--3dis}}{=} \frac{f_6^3}{f_3 f_{12}}\cdot \left(\frac{\psi(-q^9)^3}{\psi(-q^3)^4}\left(\frac{1}{w(-q^3)^2}+\frac{q}{w(-q^3)}+q^2\right)\right).
	\end{align*}
	In addition, applying the relation $(-q;-q)_{\infty}=f_2^3/(f_1f_4)$ to \eqref{def:w(q)} gives
	\begin{align*}
		w(-q) = \frac{f_6^3}{f_2}\cdot \frac{f_2^3}{f_1f_4}\frac{f_3^3f_{12}^3}{f_6^9} = \frac{f_2^2 f_3^3 f_{12}^3}{f_1 f_4 f_6^6}.
	\end{align*}
	Therefore,
	\begin{align*}
		\UU_3\left(\frac{f_2 f_6^3}{f_1 f_3 f_4 f_{12}}\right) = \frac{f_2^3}{f_1 f_{4}}\cdot \frac{\psi(-q^3)^3}{\psi(-q)^4}\frac{1}{w(-q)^2} = \left(\frac{f_2 f_6^3}{f_1 f_3 f_4 f_{12}}\right)^3,
	\end{align*}
	as required.
\end{proof}

\section{``Artificially'' constructed convolutive sequences }\label{sec:artificial}

For the moment, there is still a philosophical problem left --- \textit{must we require the generating function of our desired convolutive sequence to be an eta-product?} The answer to this question can be gleaned from the facts that the list of convolutive sequences is indeed endless and that most convolutive sequences seem to be artificial. 

There are two ways to think about the above claims.

First, we argue that convolutive sequences can be constructed directly by definition. To see this, let us assume that $(a_n)_{n\ge 0}$ is $m$-convolutive, and that we already know the values of $a_n$ for all $n$ not divisible by $m$. Then we are left to construct the values of $a_{mn}$. To begin with, we note that $a_0$ satisfies
\begin{align*}
	a_0 = a_0^m
\end{align*}
by definition, and hence $a_0$ is either $0$ or an $(m-1)$-th root of unity. Let us fix such an $a_0$. Then all $a_{mn}$ for $n\ge 1$ can be uniquely determined in an iterative way because by Definition~\ref{def:mconv},
\begin{align*}
	\sum_{n\ge 0} a_{mn} q^n = \left(\sum_{n\ge 0} a_n q^n\right)^m,
\end{align*}
which indicates that $a_{mn}$ is the coefficient of $q^n$ in the expansion of the polynomial
\begin{align*}
	\left(\sum_{k=0}^n a_k q^k\right)^m,
\end{align*}
while the information about $(a_k)_{k=0}^n$ is already known once we notice that $mn>n$ since $m\ge 2$.

On the other hand, convolutive sequences can also be constructed indirectly by certain functional equations for their generating series. Again, let $(a_n)_{n\ge 0}$ be $m$-convolutive. We write the generating function of $a_n$ as
\begin{align*}
	A(q) := \sum_{n\ge 0} a_n q^n.
\end{align*}
Then to obtain the desired $m$-convolutivity, we only need to require that $A(q)$ satisfies the functional equation:
\begin{align}\label{eq:funeqn-general}
	A(q) = A(q^m)^m + q B_1(q^m) + \cdots + q^{m-1} B_{m-1}(q^m),
\end{align}
where $B_1(q),\ldots,B_{m-1}(q)\in \mathbb{C}[[q]]$ may be either correlated to or independent of $A(q)$. As an example, we may specify the functional equation for a $2$-convolutive series as
\begin{align}\label{eq:A073707}
	A(q) = A(q^2)^2+qA(q^2)^2.
\end{align}
The coefficient sequence of this $A(q)$ is  interesting as it popped up in our search of the OEIS~\cite{OEIS}; this is indexed as sequence A073707. However, if the functional equation \eqref{eq:funeqn-general} becomes as generic as possible, then the corresponding $m$-convolutive sequence seems less meaningful.

For the OEIS sequence A073707, we rewrite \eqref{eq:A073707} as
\begin{align*}
	\frac{A(q)}{A(q^2)^2} = 1+q,
\end{align*}
which, after iterations, gives us
\begin{align}
	A(q) = \prod_{k\ge 0} (1+q^{2^k})^{2^k}.
\end{align}

More generally, we assume that $A_m(q)$ admits the functional equation
\begin{align}\label{eq:Am}
	A_m(q) = A_m(q^m)^m+qA_m(q^m)^m.
\end{align}
Then
\begin{align}
	A_m(q) = \prod_{k\ge 0} (1+q^{m^k})^{m^k},
\end{align}
which produces an $m$-convolutive sequence $(a_m(n))_{n\ge 0}$ defined by
\begin{align}
	\sum_{n\ge 0} a_m(n) q^n := A_m(q) = \prod_{k\ge 0} (1+q^{m^k})^{m^k}.
\end{align}

It remains to ask if there is any combinatorial meaning for this sequence. One of the simplest ways to interpret this sequence is by means of \emph{colored $m$-ary strict partitions}, which are partitions into powers of $m$ with each part of size $m^k$ assigned with one of $m^k$ colors for each $k\ge 0$, while no two parts share the same size and color simultaneously. For example, when $m=2$, we have $a_2(4)=5$, and the corresponding five colored binary strict partitions are
\begin{align*}
	4_1,\  4_2,\  4_3,\  4_4,\  2_1+2_2,
\end{align*}
where we mark colors by subscripts.

Let $\cA_m$ denote the set of colored $m$-ary strict partitions.

In order to combinatorially explain the $m$-convolutivity of the sequence $(a_m(n))_{n\ge 0}$ in this context, we look back at Remark~\ref{rmk:weight-halving} and make the following generalization.

\begin{lemma}
	Let $\cU$ be a set of combinatorial objects and denote by $u(n)$ the counting function which enumerates the number of objects $\pi$ in $\cU$ with weight $|\pi| = n$. Then the sequence $(u(n))_{n\ge 0}$ is $m$-convolutive if there is a bijection 
	\begin{align*}
		\begin{array}{ccc}
			[\cU]_m & \to & \cU^m\\
			\pi & \mapsto & (\lambda^{(1)},\ldots,\lambda^{(m)})
		\end{array}
	\end{align*}
	such that $|\pi| = m\cdot (|\lambda^{(1)}|+\cdots+|\lambda^{(m)}|)$, where $[\cU]_m$ is the subset of $\cU$ consisting of all members whose weight is a multiple of $m$.
\end{lemma}

\begin{theorem}
	There is a bijection
	\begin{align*}
		\begin{array}{ccc}
			[\cA_m]_m & \to & \cA_m^m\\
			\pi & \mapsto & (\lambda^{(1)},\ldots,\lambda^{(m)})
		\end{array}
	\end{align*}
	such that
	\begin{align*}
		|\pi| = m\cdot (|\lambda^{(1)}|+\cdots+|\lambda^{(m)}|),
	\end{align*}
	and
	\begin{align*}
		\ell(\pi) = \ell(\lambda^{(1)})+\cdots+\ell(\lambda^{(m)}),
	\end{align*}
	where $\ell(\pi)$ refers to the number of parts in $\pi$. As a consequence, the numbers $a_m(n)$ form an $m$-convolutive sequence.
\end{theorem}

\begin{proof}
	For $\pi \in [\cA_m]_m$, it is clear that there is no part of size $1$ because the weight of $\pi$ has to be a multiple of $m$. In fact, all parts of $\pi$ are of the form $m_{c_k}^k$ with $k\ge 1$, where the part has size $m^k$ and the color $c_k$ ranges over $1\le c_k\le m^k$. Moreover, each part $m_{c_k}^k$ appears at most once. Write
	\begin{align*}
		c_k = m (c'_k - 1) + r_{c_k},
	\end{align*}
	where $1\le r_{c_k}\le m$ and $1\le c'_k\le m^{k-1}$. In particular, this representation is unique. Now we scan through all parts $m_{c_k}^k$ in $\pi$ and construct the desired bijection by first sending $m_{c_k}^k$ to the subpartition $\lambda^{(r_{c_k})}$ in $\cA_m^m$ and then changing it to $m_{c'_k}^{k-1}$. Note that for this new part of size $m^{k-1}$, its color also ranges over $1\le c'_k\le m^{k-1}$. Clearly, by doing so, we have a partition $m$-tuple $(\lambda^{(1)},\ldots,\lambda^{(m)})\in \cA_m^m$ with $|\pi| = m\cdot (|\lambda^{(1)}|+\cdots+|\lambda^{(m)}|)$ and $\ell(\pi) = \ell(\lambda^{(1)})+\cdots+\ell(\lambda^{(m)})$. Also, this process is invertible.
\end{proof}

\section{Conclusion}\label{sec:conclusion}

A crucial step in this work revolves around an exhaustive search of the OEIS. Now we make a few comments on this experimental procedure.

It is notable that the largest initial task in our process was to obtain the entire OEIS database (which, at the time of this writing, contains more than $388\,000$ sequences).  This might sound daunting; however, it is not, thanks to the forethought of the developers of the OEIS.  One can quickly download the entire contents of the OEIS at the following link:
\begin{center}
	\url{https://oeis.org/wiki/Compressed_Versions_of_the_Database}
\end{center}
More information about this process can be found at:
\begin{center}
	\url{https://oeis.org/wiki/Download}
\end{center}

Once the contents of the database have been downloaded, it is a relatively straightforward matter to remove the A-numbers preceding each sequence, leaving one with a large file that only contains the sequence values themselves (with each sequence on its own line in the file). At that point, we chose to turn to \textit{Maple} to complete our analysis.  

After importing the data from the OEIS to a \textit{Maple} worksheet, we examined term by term if the following condition holds for each OEIS sequence $(a_n)_{n\ge 0}$:
\begin{align*}
	a_{2n} \overset{?}{=} \sum_{i=0}^n a_ia_{n-i}.
\end{align*}
This condition is equivalent to the one arising from the generating series of $a_n$, i.e., the $m=2$ case of \eqref{eq:mconv}.

Thanks to this search, we were able to identify the sequences noted in this paper. Of course, our \emph{Maple} procedure also identified other sequences in the OEIS satisfying the requirements of $2$-convolutivity.  For the most part, these were very short sequences in the OEIS that, somewhat trivially, meet the requirements of $2$-convolutivity.  In essence, we viewed these as ``false positives'' and skipped over them. The abundance of such false positives was the initial motivation for us to restrict our attention to eta-products, which further led us to the considerations in Section~\ref{sec:artificial}.

Now that we have a more concrete form of the convolutive sequences that we hope to explore, one may wonder whether we could perform a search directly on eta-products. Unfortunately, this task becomes computationally intensive. Nevertheless, there is still an observation, as pointed out by one of the referees, that may reduce the amount of search slightly. To be specific, it is sufficient to work on eta-products of \emph{weight zero}, because the unitizing operator of degree $m$ preserves the weight of a modular form, while raising such a form to the $m$-th power multiplies the weight by $m$.

As we close this work, a number of fairly obvious questions arise:
\begin{itemize}[itemsep=4pt]
	\item Can one provide combinatorial proofs of the $2$- and $3$-convolutivity of the sequences where the only proofs that we have provided rely on generating function manipulations?  

	\item Can one construct an infinite family of $2$- or $3$-convolutive sequences whose generating functions are primitive eta-products? Or, in contrast, is the list of primitive $2$- and $3$-convolutive eta-products finite?

	\item Can one find a primitive eta-product that is $m$-convolutive for a certain $m\ge 4$? (No such sequence was identified in our exhaustive search of the OEIS.)
\end{itemize}
 Our hope is that others will be motivated by the results herein to answer the above questions.

\subsection*{Acknowledgements}

Shane Chern was supported by the Austrian Science Fund grant 10.55776/F1002. Shishuo Fu was supported by the National Natural Science Foundation of China grants 12171059 and 12371336. The authors gratefully acknowledge the referees for their comments and suggestions, especially regarding the observation on weight-zero eta-products in Section~\ref{sec:conclusion}. The authors also thank Frank Garvan for his assistance in creating a straightforward way to import the OEIS sequence data into \emph{Maple}.

\bibliographystyle{amsplain}

\begin{thebibliography}{99}
	
	\bibitem{And1998}
	G. E. Andrews, \textit{The Theory of Partitions}, Cambridge University Press, Cambridge, 1998.
	
	\bibitem{ALL2002}
	G. E. Andrews, R. P. Lewis, and J. Lovejoy, Partitions with designated summands, \textit{Acta Arith.} \textbf{105} (2002), no. 1, 51--66.
	
	\bibitem{BO15}
	N. D. Baruah and K. K. Ojah, Partitions with designated summands in which all parts are odd, \textit{Integers} \textbf{15} (2015), Paper No. A9, 16 pp.
	
	\bibitem{Ber91}
	B. C. Berndt, \textit{Ramanujan's Notebooks. Part III}, Springer-Verlag, New York, 1991.
	
	\bibitem{Chen13}
	W. Y. C. Chen, K. Q. Ji, H.-T. Jin, and E. Y. Y. Shen, 
	On the number of partitions with designated summands, \textit{J. Number Theory} \textbf{133} (2013), no. 9, 2929--2938.
	
	\bibitem{ChernSellers}
	S. Chern and J. A. Sellers, An infinite family of internal congruences modulo powers of $2$ for partitions into odd parts with designated summands, \textit{Acta Arith.} \textbf{215} (2024), no. 1, 43--64.
	
	\bibitem{CL2004}
	S. Corteel and J. Lovejoy, Overpartitions, \textit{Trans. Amer. Math. Soc.} \textbf{356} (2004), no. 4, 1623--1635.
	
	\bibitem{FuSel}
	S. Fu and J. A. Sellers, A refined view of a curious identity for partitions into odd parts with designated summands, \textit{Discrete Math.} \textbf{348} (2025), no. 12, Paper No. 114620, 13 pp.
	
	\bibitem{Gla1883}
	J. W. L. Glaisher, A theorem in partitions, \textit{Messenger Math.} \textbf{12} (1883), 158--170.
	
	\bibitem{Hir}
	M. D. Hirschhorn, \textit{The Power of $q$. A Personal Journey}, Springer, Cham, 2017.
	
	\bibitem{Lew1984}
	R. P. Lewis, A combinatorial proof of the triple product identity, \textit{Amer. Math. Monthly} \textbf{91} (1984), no. 7, 420--423.

	\bibitem{Lov2005}
	J. Lovejoy, A theorem on seven-colored overpartitions and its applications, \textit{Int. J. Number Theory} \textbf{1} (2005), no. 2, 215-224.
	
	\bibitem{Pak2006}
	I. Pak, Partition bijections, a survey, \textit{Ramanujan J.} \textbf{12} (2006), no. 1, 5--75.
	
	\bibitem{Sel2024_INTEGERS}  
	J. A. Sellers, New infinite families of congruences modulo powers of $2$ for $2$-regular partitions with designated summands, \textit{Integers} \textbf{24} (2024), Paper No. A16, 17 pp.
	
	\bibitem{Shen}
	E. Y. Y. Shen, A crank of partitions with designated summands, \textit{Ramanujan J.} \textbf{57} (2022), no. 2, 785--802.
	
	\bibitem{OEIS}
	N. J. A. Sloane, \textit{On-Line Encyclopedia of Integer Sequences}, available at \url{http://oeis.org}.
	
	\bibitem{Ver1988}
	A. M. Vershik, A bijective proof of the Jacobi identity, and reshapings of the Young diagrams, \textit{J. Soviet Math.} \textbf{41} (1988), no. 2, 889--891.
	
	\bibitem{Xia}    
	E. X. W. Xia, Arithmetic properties of partitions with designated summands, \textit{J. Number Theory} \textbf{159} (2016), 160--175.
	
\end{thebibliography}

\end{document}